\documentclass[12pt]{amsart}
\usepackage[left=3.5cm,tmargin=2cm,bmargin=2.5cm,centering]{geometry}
\usepackage{hyperref}
\usepackage[utf8]{inputenc}
\usepackage[initials]{amsrefs}  
\usepackage{enumerate}
\usepackage{bm,amsmath,amssymb,amsthm}
\usepackage[all]{xy}
\usepackage{stmaryrd}
\usepackage{amsfonts}
\theoremstyle{plain}
\newtheorem{theorem}{Theorem}[section]

\newtheorem{corollary}[theorem]{Corollary}

\newtheorem{proposition}[theorem]{Proposition}
\newtheorem{lemma}[theorem]{Lemma}
\newtheorem{definition}[theorem]{Definition}

\newtheorem*{proposition*}{Proposition}
\newtheorem*{lemma*}{Lemma}
\newtheorem*{corollary*}{Corollary}
\theoremstyle{definition}
\newtheorem*{examples*}{Examples}
\newtheorem*{example*}{Example}

\theoremstyle{remark} 

\newtheorem*{remarks*}{Remarks}
\newtheorem*{remark*}{Remark}
\numberwithin{equation}{section}

\newcommand{\am}[1]{\mdede{#1}{0}{0}{1}}
\newcommand{\km}[1]{\mdede{#1}{0}{#1}{1}}
\newcommand{\nm}[1]{\mdede{1}{#1}{0}{1}}

\newcommand{\qtext}[1]{\quad\text{#1}}
\newcommand{\mdemi}{\tfrac{1}{2}}
\newcommand{\Mdemi}{\frac{1}{2}}
\newcommand {\pnorm}[1]   {\left\lVert #1 \right\rVert}
\newcommand{\BmZ}{\mathbb{Z}}
\newcommand{\FmH}{\mathfrak{H}}
\newcommand{\abs}[1]{\ensuremath{\left|#1\right|}}

\usepackage{dsfont}  
\newcommand{\Mun}{\mathds{1}}
\newcommand{\BmR}{\mathbb{R}}
\newcommand{\BmN}{\mathbb{N}}
\newcommand{\SB}{\backslash}
\newcommand{\BmA}{\mathbb{A}}
\newcommand{\BmQ}{\mathbb{Q}}
\newcommand{\TmN}{\mathbf{N}}
\DeclareMathOperator {\GL} {GL}
\newcommand{\FmN}{\mathfrak{N}}

\newcommand{\Tme}{\mathbf{e}}
\newcommand{\Fma}{\mathfrak{a}}
\DeclareMathOperator {\Mvol} {vol}
\DeclareMathOperator {\MIm}  {Im}
\DeclareMathOperator {\Mim}  {\Im m}
\newcommand{\Q}{\mathbb{Q}}
\newcommand{\CmW}{\mathcal{W}}
\newcommand{\Mdede}[4]{
\begin{pmatrix}
#1&#2  \\
#3&#4 \\
\end{pmatrix}
}

\DeclareMathOperator {\Mcodim}{codim}
\newcommand{\set}[1]{\left\{#1\right\}}
\DeclareMathOperator{\MSt}{St}
\newcommand{\Lquote}[1]{``#1"}
\newcommand{\Tmk}{\mathbf{k}}
\newcommand{\mdede}[4]{
\left(
\begin{smallmatrix}
#1&#2  \\
#3&#4 
\end{smallmatrix}
\right)
}

\newcommand{\Tmw}{\mathbf{w}}
\newcommand{\BmC}{\mathbb{C}}
\DeclareMathOperator {\MAd}   {Ad}

\begin{document}
\title[]{Large values of modular forms}
\author[]{Nicolas Templier}   
\address{Department of Mathematics, Fine Hall, Washington Road, Princeton, NJ 08544-1000.}
\email{templier@math.princeton.edu}
\date{\today}
\makeatletter
\@namedef{subjclassname@2010}{%
  \textup{2010} Mathematics Subject Classification}
\makeatother

\begin{abstract} We show that there are primitive holomorphic modular forms $f$ of arbitrary large level $N$ such that $ \abs{f(z)} \gg N^{\frac14}$ for some $z\in \FmH$. Thereby we disprove a folklore conjecture that the $L^\infty$-norm of such forms would be as small as $N^{o(1)}$.
\end{abstract}

\maketitle




\section{Introduction}
L
et $N\ge 1$ and $\Gamma_0(N) \subset SL_2(\BmZ)$ be the Hecke congruence subgroup of level $N$. Let $\chi$ be a Dirichlet character modulo $N$ and $S_2(N,\chi)$ the vector space of holomorphic cusp forms of weight two, level $N$ and nebentypus $\chi$. Let $S_2^*(N,\chi)$ be the set of primitive forms
\begin{equation}\label{def:f}
  f(z)= y \sum^{\infty}_{n=1} a_n n^{\frac12} e^{2i\pi nz},\quad z=x+iy\in \FmH.
\end{equation}
Recall that primitive forms satisfy $a_1=1$ and are eigenvalues of all Hecke operators $T_n$, $n\ge 1$, see~e.g~\cite{book:Miy}*{\S4.6}. The $L$-function $L(s,f)=\sum\limits^\infty_{n=1} \dfrac{a_n}{n^s}$ is entire and admits a functional equation and an Euler product.

The absolute value $\abs{f(z)}$ is $\Gamma_0(N)$-invariant and bounded on $\FmH$. It is of great interest to estimate the sup-norm $\pnorm{f}_\infty$. Such estimates are closely related to quantum ergodicity and entropy bounds~\cite{BL03}, they occur in the subconvexity problem for $L$-functions~\cite{HM06}, in the distribution of zeros of modular forms~\cites{GS12,GRS12} and in the study of Arakelov invariants of $X_0(N)$~\cites{AU95,MU98,JK06}.

In~\cite{AU95}*{Thm.\,A} for $N$ square-free and in~\cite{JK04}*{Cor.\,3.2} in general, it is proven that $\pnorm{f}_\infty \ll_\epsilon N^{\frac12+\epsilon}$ for all $\epsilon>0$. Blomer--Holowinsky~\cite{BH09} show that $\pnorm{f}_\infty \ll N^{\Mdemi-\delta}$ for $N$ square-free and $\delta<25/914$. Further improvements in~\cite{Temp:sup} and~\cites{Temp:hybrid,Temp:level-III} give $\delta<1/6$, thus 
\begin{equation} 
  \label{HT}
  1\ll \pnorm{f}_\infty \ll_\epsilon N^{\frac13+\epsilon}.
\end{equation}

\subsection{A folklore conjecture}\label{sec:intro:conj}
Conversely, how large can $\abs{f(z)}$ be? A folklore conjecture is that 
\begin{equation}\label{folklore}\tag{A}
  \pnorm{f}_\infty \stackrel{?}{\ll_\epsilon} N^{\epsilon}, \qtext{for all $\epsilon>0$}.
\end{equation}
See~e.g.~\cite{BH09}*{Eq.\,(5)}\footnote{\cite{BH09} record this conjecture with the two reservations that it is optimistic and that it is necessary that $f$ be an Hecke eigenform since it fails for holomorphic Poincar\'e series.} and \cite{JK04}*{Rem.\,3.3}. There have been partial evidences in favor of the estimate~\eqref{folklore}: it is true on average over an orthonormal basis of $S_2(N,\chi)$; it would imply the Lindel\"of hypothesis in the level aspect; it directly applies to old forms; and it is consistent with square-root cancellation of sums of Fourier coefficients.

\subsection{Main result}
We show that the estimate~\eqref{folklore} is false:
\begin{theorem}\label{th:main}
  There are primitive forms $f\in S_2^*(N,\chi)$ of arbitrary large level $N$ such that $\pnorm{f}_\infty \gg N^{\frac{1}{4}}$. 
\end{theorem}

\begin{remarks*}
  \begin{enumerate}[(i)]
	\item The multiplicative constant is absolute and indeed we shall produce an explicit value such as $\pnorm{f}_\infty \ge (2\pi e)^{-1} N^{\frac{1}{4}}$, where $\frac{1}{2\pi e}= 0.0585498\ldots$. 
	  \item Our method is specific to levels $N$ which are not square-free and to the non-compact case. It remains an open problem whether similar lower bounds exist for square-free levels or for compact surfaces. 
	  \item One can see a similar phenomenon for Hecke eigenfunctions of the quantized cat map as shown by Olofsson~\cite{Olof:large}. In fact the same exponent $N^{\frac14}$ occurs where $N$ denotes the underlying quantum multiplicity. Also it is interesting to compare Proposition~\ref{prop:neigh} below with~\cite{Olof10} and Theorem~\ref{th:levels} below with \cite{Olof:large}*{Thm.\,3.2}.
  \end{enumerate} 
\end{remarks*}

		Conditional under the GRH it is shown by Lau~\cite{Lau10:omega}*{Thm.\,1} that there are primitive forms $f$ of arbitrary large level $N$ such that
\begin{equation*} 
  \pnorm{f}_\infty \ge \exp\left( (\Mdemi+o(1))\sqrt{\frac{\log N}{\log \log N}} \right).
\end{equation*}
The method of proof is via the Hecke integral (see~\eqref{hecke} below) and Soundararajan's resonance method~\cite{Sound:extreme} to produce omega results for central values of $L$-functions. There are analogous results of Mili{\'c}evi{\'c}~\cite{Mili:large} in the eigenvalue aspect.
 
Other results on lower bounds for the sup-norm of automorphic forms in general may be found in Rudnick--Sarnak~\cite{RS94:eigenstate} and Lapid--Offen~\cite{LO07}. 
The mechanism there is quite different since large values arise from functorial lift from other groups. 
In our $\GL(2)$ situation dihedral forms could play a similar role but we don't know whether they exhibit special behavior with respect to sup-norms. 

Here we shall establish the following which surprisingly shows that the primitive forms $f\in S_2^*(N,\chi)$ with large sup-norm are actually abundant: 
\begin{theorem}\label{th:levels}
For any primitive character $\chi$ (mod $N$) and any form $f\in S_2^*(N,\chi)$,
\[
\pnorm{f}_\infty \gg \prod_{p^c || N} p^{\frac12 \lfloor \frac{c}{2} \rfloor}
\]
where $p^c || N$ for a prime $p$ and an integer $c\ge 0$ means that $p^c\mid N$ and $p^{c+1}\nmid N$.
\end{theorem}

 Since $\chi$ is primitive, each form $f\in S_2^*(N,\chi)$ is twist-minimal in the sense that for any Dirichlet character $\eta$ the level of the form $f\otimes \eta$ is at least $N$ (see Lemma~\ref{lem:ps}).

\begin{examples*}\begin{enumerate}[(i)] 
  \item If $N$ is an even power of a prime or more generally a perfect square, then the right-hand side is equal to $N^{\frac14}$. Thus we have $\pnorm{f}_\infty \gg N^{\frac14}$ for \emph{all} primitive $\chi$ (mod $N$) and \emph{all} $f\in S_2^*(N,\chi)$ if $N$ is a  square.  
  \item If $N$ is the third power of a square-free integer, then the right-hand side is equal to $N^{\frac16}$. If $N$ is square-full in the sense that $p\mid N$ implies $p^2\mid N$ for all primes $p$, then the right-hand side is between $N^{\frac16}$ and $N^{\frac14}$.
\end{enumerate}
\end{examples*}

One may now wonder what the true size of $\pnorm{f}_\infty$ should be. A possible answer is to expect a purity phenomenon~\cite{Sarn:Morawetz}: the accumulation points of the ratio $\dfrac{\log \pnorm{f}_\infty}{\log N}$ could be restricted to a certain set $E$ of exponents.
Purity conjectures are put forward by Clozel and Sarnak~\cite{Sarnak:GRC} in analogy with Deligne's purity theorem in algebraic geometry. Purity can for example accommodate the exceptions to Ramanujan type bounds. The question is to determine what these accumulation points should be. What is known so far is that $E\subset [0,\frac{1}{3}]$ by~\eqref{HT} and by Theorem~\ref{th:main} that the intersection $E\cap [\frac{1}{4},\frac{1}{2}]$ is non-empty.

Since in Theorem~\ref{th:levels} we have that $\frac{1}{2c}\lfloor \frac{c}{2} \rfloor$ takes infinitely many values below $\frac{1}{4}$, this suggests  caution when speculating on the purity of $\dfrac{\log \pnorm{f}_\infty}{\log N}$. For example it could be that the accumulation set $E$ is infinite although our argument is not yet conclusive since a corresponding upper bound for $\pnorm{f}_\infty$ is not known.

\subsection{Hilbert modular forms}
Our results and methods of proof generalize to automorphic forms on $\GL(2)$. Let $F$ be a totally real number field of degree $d$. 
We show that there are Hilbert modular forms $f$ of weight $(k_1,\cdots,k_d)$ and level $\FmN$ such that
				\begin{equation*} 
				  \pnorm{f}_\infty \gg (k_1\cdots k_d)^{\frac14} \TmN_{F/\BmQ}(\FmN)^{\frac{1}{4}} \pnorm{f}_2.
				\end{equation*}
The general result is stated below in Theorem~\ref{th:hilbert} and includes Maass forms. Let $\pi=\otimes_v \pi_v$ be an unitary cuspidal automorphic representation of $\GL(2,\BmA_F)$.
If the place $v$ is non-archimedean, then $\pi_v$ has level $\FmN_v$. 
If the place $v$ is archimedean, then either $\pi_v$ is a discrete series (resp. limit of discrete series) of some weight $k\ge 2$ (resp. of weight $k=1$) and we denote by $h$ the number of such places; or $\pi_v$ is a principal or complementary series of spectral parameter $r > 0$ and we denote by $m$ the number of such places. Thus $h+m=d$. 
The representation $\pi$ is generated by a unique Hilbert--Maass newform $f\in \pi$.
Hilbert--Maass forms belongs to $L^2_\omega(\GL(2,F)\SB \GL(2,\BmA_F))$ where $\omega$ is the central character.
\begin{theorem}\label{th:hilbert}
  Fix $\epsilon>0$ and a totally real number field $F$ of degree $d$.
  For all but finitely many tuples of integers $k_1,\ldots, k_h, t_1,\ldots t_m \in \BmN_{\ge 1}$ and square integral ideals $\FmN$, there are Hilbert-Maass newforms $f$ of level $\FmN$, archimedean type $(h,m)$ with weights $(k_1,\ldots,k_h)$ and spectral parameters in the respective intervals $[t_1-1,t_1], \ldots, [t_m-1,t_{m}]$ such that 
\[
\pnorm{f}_\infty \gg (t_1 \cdots t_m)^{\frac16-\epsilon} (k_1\cdots k_h)^{\frac14-\epsilon} \TmN_{F/\BmQ}(\FmN)^{\frac14-\epsilon} \pnorm{f}_2.
\]
 The multiplicative constant may depend only on $F$ and $\epsilon>0$.
\end{theorem}

\subsection{Whittaker models and newvector theory}

 We will introduce a local invariant to study large values of modular forms, the key novelty being Proposition~\ref{prop:largew} below. 
It is shown by Iwaniec--Sarnak in the corrigendum in~\cite{Sarn:Morawetz} of~\cite{IS95}*{Lem.\,A.1} that a Hecke--Maass form $f$ on $SL(2,\BmZ)\SB \FmH$ of large spectral parameter $r$ satisfies $\pnorm{f}_\infty \gg_\varepsilon r^{\frac16-\varepsilon}\pnorm{f}_2$. This is related to the asymptotic behavior of the $K$-Bessel function in the transition range. 

This paper shows that 
large values are a general phenomenon that occurs not only in the eigenvalue aspect but also in the level aspect. In the level aspect one might hope that exponential sums will play the role of special functions. Thus our starting point has been to draw in the analogy between archimedean and non-archimedean places and to look for a non-archimedean analogue of the transition range of the $K$-Bessel function.  Such analogy isn't obvious but indeed exists as we now explain.

We first isolate the following as the key of the lower bound~\cite{Sarn:Morawetz} for Hecke--Maass forms:
\begin{equation} \label{kirillov-transition}
\max_{y>0}{
y^{\frac12}\abs{K_{ir}(y)} \asymp r^{\frac16}
} e^{-\frac{\pi r}{2}}, \qtext{as $r\to \infty$}.
\end{equation}
See \S\ref{sec:intro:ideas} and \S\ref{sec:pf:hilbert} for details.

Now $y^{\frac12}K_{ir}(y)$ is the spherical vector in the Kirillov model of a principal series representation of $\GL(2,\BmR)$. So in analogy it is natural to investigate the newvector in the Kirillov model of a ramified representation $\pi$ of $\GL(2,\BmQ_p)$. Explicit formulas are known~\cites{Cass73,book:moduII:deligne}, namely we find 
\begin{equation*}
 \eta(y)\abs{y}^{\frac12}\Mun_{\BmZ_p}(y)
\quad \text{or}
\quad 
 \Mun_{\BmZ_p^\times}(y)
\end{equation*}
depending on the representation (here $\eta$ is an unramified character of $\BmQ_p^\times$).
The analogy with the archimedean $y^{\frac12}K_{ir}(y)$ is good; we observe the same oscillatory behavior as $y\to 0$ and the same rapid decay as $y\to \infty$.

However there is a difference which is that the Kirillov newvector in the non-archimedean case doesn't exhibit any transition range as a function in the $y$-variable.
In particular there is no analogue of~\eqref{kirillov-transition}, the Kirillov function sharply goes from one to zero between $\abs{y}= 1$ and $\abs{y}>1$.

Our solution is to introduce a local invariant $h(\pi)$ of generic representations of $\GL(2,\BmQ_p)$. It is defined as the maximum value of a newvector $W_{\circ}$ in the Whittaker model. We have that $W_{\circ}(\Mdede{y}{0}{0}{1})$ is the newvector in the Kirillov model as above. Our main point is that $W_{\circ}$ achieves large values outside of the diagonal. Thus the correct analogue of~\eqref{kirillov-transition} and the transition region in the non-archimedean case is to be found in the Whittaker function as opposed to the Kirillov function.  We establish the following result in Section~\ref{sec:gauss} which implies $h(\pi)=p^{\frac12}$ for certain type of representations: 
\begin{proposition}\label{prop:largew}
  For a twist-minimal unitary principal series representation of $\GL(2,\BmQ_p)$ of conductor $p^2$ the newvector $W_\circ$ in the Whittaker model satisfies
  \begin{equation*} 
	\max_{g\in \GL(2,\BmQ_p)} \abs{W_\circ(g)} = p^{\frac12} \abs{W_\circ(\Tme)}.
  \end{equation*}
\end{proposition}

\subsection{Ideas of proof}\label{sec:intro:ideas}



Let $f\in S_2^*(N,\chi)$ be a primitive newform of level $N$.
For each cusp $\Fma$ of $\Gamma_0(N)\SB \FmH$ we can consider the Fourier expansion at $\Fma$ and integrate $f$ against closed horocycles. The number of inequivalent cusps of $\Gamma_0(N)\SB \FmH$ is 
$ 
  \sum\limits_{ab=N} \varphi( (a,b) )
$
where  $\varphi$ is Euler function and $(a,b)$ is the g.c.d of $a$ and $b$. If $N$ is square-free then there are $2^{\omega(N)}$ cusps where $\omega(N)$ is the number of prime factors of $N$.

Next we take into account the group of Atkin--Lehner involutions~\cite{AL70} because if a cusp $\Fma$ is conjugate to $i\infty$ by an Atkin--Lehner involution, then the Fourier expansions of $f$ at $\Fma$ and $i\infty$ are directly related. The cusp $\Fma$ does carry the same information as the cusp $i\infty$ so that the same asymptotics arise when investigating periods against closed horocycles at such cusps. The number of cusps is greater than the number $2^{\omega(N)}$ of Atkin--Lehner involutions if and only if $N$ is not square-free.
 
The conclusion so far is that one should investigate the case when $N$ is \emph{not} square-free and look at closed horocycles around cusps of $\Gamma_0(N)\SB \FmH$ which are \emph{not} conjugate to $i\infty$ by the group of Atkin--Lehner involutions. Fortunately this move to a somehow inextricable situation will actually succeed! 

At this stage we bring into play representation theory, viewing $f$ as generating a cuspidal automorphic representation $\pi\simeq \otimes_p \pi_p$.
Our key point is to work with the invariants $h(\pi_p)$ attached to the representations $\pi_p$.
Among other things it captures the local harmonic analysis underlying the solution of the problem. The idea should apply more generally to several kind of periods of automorphic forms as it primarily relies on factorization of period integrals in local functionals.

The method will be formalized in Section~\ref{sec:wh}. We start with a simple inequality in Lemma~\ref{lem:lower} which is a variant of Hecke classical bound~\cite{hecke-werke}*{Satz\,8~p.\,484}. This reduces lower bounds for $\pnorm{f}_\infty$ to lower bounds for $h(\pi_p)$. Then we are in position to apply the Proposition~\ref{prop:largew}. The global analysis above saying that $\Fma$ is not conjugate to $i\infty$ translates into the local fact that the newvector $W_\circ$ achieves its maximum at a matrix $g\in \GL(2,\Q_p)$ that is neither diagonal nor anti-diagonal.

   The value $h(\pi_p)=p^{\frac12}$ equals the fourth-root of the conductor $p^2$ of $\pi_p$, which explains the exponent $\frac14$ in Theorem~\ref{th:main}. It seems to be the largest among all representations of $\GL(2,\BmQ_p)$ in which case Theorem~\ref{th:main} would provide the best lower bound that can be attained by our method.

The proof of Theorem~\ref{th:hilbert} is based on the same ideas by including the uniformity in the parameters at infinity, see Section~\ref{sec:pf}. The introduction of the local invariant $h(\pi)$ allows to make precise the analogy between archimedean and non-archimedean places. 
Indeed we can explain the respective exponents $\frac14$ in the level aspect, $\frac16$ in the eigenvalue aspect~\cite{Sarn:Morawetz} (Maass forms), $\frac14$ in the weight aspect~\cite{Xia07} (holomorphic forms), by showing that $h(\pi)=N^{\frac14}$ for certain non-archimedean $\pi$ of conductor $N=p^2$, that $h(\pi) \asymp r^{\frac16}$ for $\pi$ an archimedean principal series of parameter $r$ and that $h(\pi) \asymp k^{\frac14}$ for $\pi$ a discrete series of weight $k$.

\section{Complements to the main results}\label{sec:complement}

We explore variants and corollaries as well as relation to other questions.
The goal is to revisit some previous estimates in light of our results.
The content of this section is summarized as follows:
\begin{enumerate}[(i)]
\item For a given square level $N$ and a primitive character $\chi$ (mod $N$) the Theorem~\ref{th:main} says that $\pnorm{f}_\infty \gg N^{\frac{1}{4}}$. Surprisingly there is a single point $z_\chi\in \FmH$ depending only on $\chi$ at which these large values are achieved. Namely we have
\[
\pnorm{f}_\infty \ge 
\abs{f(z_\chi)}\gg N^{\frac14}\]
for all $f\in S_2^*(N,\chi)$, see Theorem~\ref{th:point}.
\item For any $4<r\le \infty$ we have that $\pnorm{f}_r\to \infty$ as $N\to \infty$ (again assuming that $N$ is a square and $\chi$ is primitive, see Proposition~\ref{prop:neigh}).
\item We establish  upper and lower estimates for  Wilton partial sums $\sum\limits_{m\le M} a_m e(mx)$ where $x\in (0,1)$ and the Fourier coefficients $a_m$ are given by~\eqref{def:f}.
\end{enumerate}

\subsection{Mean value estimates.}\label{conj:mean} Let 
\begin{equation}\label{def:M} 
	M(z):=\sum_{g\ o.b.\,S_2(N,1)} \abs{g(z)}^2, \quad z\in \FmH
  \end{equation}
  where the sum is over an orthonormal basis of $S_2(N,1)$ for the Petersson inner-product. Note that $\dim S_2(N,1)=N^{1+o(1)}$ and by definition 
  \begin{equation}\label{mean-M} 
  \frac{1}{\Mvol(\Gamma_0(N)\SB \FmH)} \int_{\Gamma_0(N)\SB \FmH}
  M(z) \frac{dxdy}{y^2} = \frac{1}{4\pi} + O(N^{-1+o(1)}).
\end{equation}
It is shown by Jorgenson--Kramer~\cite{JK04} that 
\begin{equation} \label{JK}
  \pnorm{M}_\infty \ll 1.
\end{equation}
The estimate is quite robust since it applies more generally to finite coverings of a given non-compact Riemann surface.

It follows from~\eqref{JK} that the upper bound $\pnorm{f}_\infty \ll N^{\frac12+\epsilon}$ holds for all primitive forms $f\in S^*_2(N,1)$. Indeed $g:=\dfrac{f}{(f,f)^{\frac12}}$ is an orthonormal vector for the Petersson inner product and $(f,f)=N^{1+o(1)}$ by results of Iwaniec and Hoffstein--Lockhart.

Another deep estimate for $M(z)$ is in the work of Michel--Ullmo~\cite{MU98}*{Thm.\,1.6}: 
\begin{equation}\label{MU}
 \frac{1}{\Mvol(\Gamma_0(N)\SB \FmH)} \int_{\Gamma_0(N)\SB \FmH}
  M(z)^2 \frac{dxdy}{y^2} = \frac{1}{(4\pi)^2} + O(N^{-\delta}),
\end{equation}
for $N$ square-free and some $\delta>0$. 

The previously conjectured estimate~\eqref{folklore} is compatible with~\eqref{JK} and~\eqref{MU}, e.g. it would have followed from~\eqref{JK} \emph{if} the mass $\abs{f(z)}^2$ \emph{were} equally distributed among all primitive forms in $S^*_2(N,1)$, see~\cite{JK04}*{Rem.\,3.3}.

We now know from Theorem~\ref{th:main} that~\eqref{folklore} doesn't hold in general.
The conclusion one may draw  is that the values $f(z)$ for different $f\in S_2^*(N,\chi)$ fluctuate much more than one might originally expect. In fact we shall derive a precise statement by considering the mean value estimate~\eqref{JK} for general nebentypus. Thus for $\chi$ an even Dirichlet character, let
\begin{equation}\label{def:Mchi}
  M_\chi(z) := \sum_{g\ o.b.\ S_2(N,\chi)} \abs{g(z)}^2, \quad z\in \FmH.
\end{equation}
The sum is over an orthonormal basis of $S_2(N,\chi)$ and the asymptotic~\eqref{mean-M} still holds.
If $\chi=\Mun$, then $M_\chi(z)=M(z)$.
Consider also the average over characters 
\begin{equation}\label{def:MGamma1}
  M_{\Gamma_1(N)}(z):=\frac{2}{\phi(N)} \sum_{\chi\ (N)} M_\chi(z), \quad z\in \FmH.
\end{equation}
The number $\frac{\phi(N)}{2}$ of even Dirichlet characters modulo $N$ equals the index ${[\Gamma_0(N):\Gamma_1(N)]}$. 

The function $M_{\Gamma_1(N)}(z)$ (resp. $M(z)$) can be interpreted as the ratio of the Arakelov metric and the hyperbolic metric for $\Gamma_1(N)\SB \FmH$ (resp. $\Gamma_0(N)\SB \FmH$). It is also shown in~\cite{JK04} that $\pnorm{M_{\Gamma_1(N)}}_\infty \ll 1$ (their result applies to any finite cover of a given Riemann surface). This implies that for all $\epsilon>0$,
\begin{equation*}
  \frac{1}{\Mvol(\Gamma_1(N)\SB \FmH)} \sup\limits_{z\in \FmH} 
  \sum_{\chi\ (N)} \sum_{f\in S_2^*(N,\chi)}
  \abs{f(z)}^2 \ll_\epsilon N^\epsilon.
\end{equation*}

On the other hand we have seen in Theorem~\ref{th:main} that there are primitive forms $f\in S_2^*(N,\chi)$ of arbitrary large level such that $\pnorm{f}_\infty \gg N^{\frac14}$. This might be interpreted by saying that for certain $z\in \FmH$ the absolute value $\abs{f(z)}$ fluctuate much with $f$.

Now it is interesting to ask about the size of $\pnorm{M_\chi}_\infty$ itself, for $\chi$ primitive. Our first observation is that the proof in~\cite{JK04} that $\pnorm{M}_\infty \ll 1$ doesn't extend to $M_\chi$. Indeed the argument on p.1275 of~\cite{JK04} uses the positivity of the heat kernel $K(t;z,\gamma z)$ which for non-trivial nebentypus should be replaced by $\chi(\gamma) K(t;z,\gamma z)$, destroying the positivity. Our second observation is that the proof in~\cite{MU98} doesn't apply either, one reason is that it is restricted to square-free levels. 

In fact we have the following result which clarifies the situation because it shows that $\pnorm{M_\chi}_\infty$ can become surprisingly large for non-trivial nebentypus. 
\begin{proposition}\label{prop:Mchi}
  For all even primitive Dirichlet character $\chi$ modulo $N=p^2$ with $p$ prime, we have $\pnorm{M_\chi}_\infty \gg_{\epsilon} N^{\frac12-\epsilon}$.
\end{proposition}

Not only the individual forms may assume large values but also their entire average $M_\chi(z)$ in~\eqref{def:Mchi}! Thus the respective arguments in~\cite{MU98} and~\cite{JK04}*{p.1275} don't generalize to $M_\chi$ for a good reason.

\begin{proof}[Proof of Proposition~\ref{prop:Mchi}]
It follows from Theorem~\ref{th:point} below that 
$g(z_\chi) \gg N^{-\epsilon}$ where $g:=\frac{f}{(f,f)^{\frac12}}$ and $f\in S_2^*(N,\chi)$.
Since $z_\chi\in \FmH$ depends only on $\chi$ (and not on $g$), we have $M_\chi(z_\chi) \gg N^{\frac12-\epsilon}$ which concludes the proof.
\end{proof}

A comparison may also be drawn with~\eqref{def:MGamma1}: since $\pnorm{M_{\Gamma_1(N)}}_\infty \ll 1$, the absolute value $\abs{M_\chi(z)}$ fluctuate much with $\chi$ for certain $z\in \FmH$. This is consistent with~\eqref{def:zchi} below in which $z_\chi$ shall vary with $\chi$.

\subsection{A specific CM-point}
We refine the statement of Theorem~\ref{th:main} by finding a specific point $z\in \FmH$ such that $\abs{f(z)}\gg N^{\frac14}$. 
\begin{theorem}\label{th:point}
 Let $\chi$ be an even primitive Dirichlet character modulo $p^2$. There is a unique $b\in (\BmZ/p\BmZ)^\times$ such that $\chi(1-pz)=e^{2\pi i \frac{bz}{p}}$ for all $z\in \BmZ$. Let $a\in (\BmZ/p\BmZ)^\times$ be the multiplicative inverse: $ab\equiv 1 \pmod{p}$. 

For all primitive forms $f\in S_2^*(p^2,\chi)$, we have $\abs{f(z_\chi)}\gg p^{\frac12}$ where
 \begin{equation}\label{def:zchi}
   z_\chi:=\frac{a}{p}+\frac{i}{p^3} \in \FmH.
 \end{equation} 
\end{theorem}

\begin{remarks*}\begin{enumerate}[(i)] 
  \item When establishing $\pnorm{f}_\infty \ll N^{\frac13}$ for $N$ square-free in~\cites{BH09,Temp:sup,Temp:harcos,Temp:level-III}, the critical region is $N^{-1} \ll \Im(z)\ll N^{-\frac23}$. The lower bound $N^{-1} \ll \Im(z)$ follows by an application of Atkin--Lehner operators and the fact that $N$ is square-free. In comparison, $N^{-\frac32}=\MIm(z_\chi)$ in~\eqref{def:zchi} above is much smaller, where $N:=p^2$. 
  \item One has $\MIm(z_\chi)\ge \MIm(\gamma z_\chi)$ for all $\gamma \in \Gamma_0(p^2)$. Also if we let $z'_\chi:= \frac{-1}{p^2z_\chi}$, then
	$\frac{1}{p^3}>\MIm(\gamma z'_\chi) $. Actually it is possible to verify the following: 
\begin{equation*}
	  \frac{1}{p^7} \ll \frac{1}{p^3} - \max_{\gamma \in \Gamma_0(p^2)} \MIm(\gamma z'_\chi) \ll \frac{1}{p^5}.
	\end{equation*}
\item Since $z_\chi$ belongs to $\BmQ(i)$ it is a CM-point. It would be interesting to carry out the analysis of the local zeta integral at $p$ attached to Waldspurger formula for $L(\mdemi,f\times \theta)$, where $\theta$ is a Gr\"ossencharacter of finite order of $\Q(i)$ ramified at $p$.
\end{enumerate}
\end{remarks*}

The proof of Theorem~\ref{th:point} will be given in~\S\ref{sec:pf:point}. With more work we also establish lower bounds for other $L^r$-norms. If $2\le r \le \infty$, let
\begin{equation*}
  \pnorm{f}^r_r:= \frac{1}{\Mvol(\Gamma_0(N)\SB \FmH)}
  \int_{\Gamma_0(N)\SB \FmH} \abs{f}^{r} \frac{dxdy}{y^2}.
\end{equation*}
\begin{proposition}\label{prop:neigh}   
(i) For any fixed $0<\delta<\frac 12$ we have $\abs{f}\gg p^\delta$ on a neighborhood of $z_\chi$ of volume at least $p^{\frac12-\delta}$, with $z_\chi$ as in Theorem~\ref{th:point}.
 
(ii) Thus for all $r>4$, we have that
  $	\dfrac{\pnorm{f}_r}{\pnorm{f}_2} \to \infty	$
	as $p\to \infty$ uniformly for all $f\in S_2^*(p^2,\chi)$ and primitive character $\chi$ modulo $p^2$.
\end{proposition}

\subsection{Square-root cancellation and uniform Wilton estimates}\label{conj:lindelof} Integral representations of special values of $L$-functions yield relations between the sup-norm of forms and the subconvexity problem. Let $\theta$ be a fixed dihedral $\GL(2)$ form induced from a Hecke character of finite order on a quadratic field. The Lindel\"of hypothesis states that $L(\mdemi,f\times \theta)\ll_{\epsilon,\theta} N^\epsilon$, which is of course very reliable because it would follow from the GRH for $L(s,f\times \theta)$. 

The conjectured estimate~\eqref{folklore} would have implied the Lindel\"of hypothesis for $L(\Mdemi,f\times \theta)$ as follows from the Waldspurger period formula, see~\cite{BH09} and~\cite{BM11}*{\S6}. 
Conversely the Lindel\"of hypothesis for $L(\Mdemi,f\times \theta)$  implies a bound $\abs{f(z)}\ll_\epsilon N^\epsilon$ for \emph{fixed} CM points $z\in \FmH$.
For sup-norm bounds in the eigenvalue aspect the relation to the subconvexity problem is discussed in~\cite{IS95}*{Rem.\,D} and~\cite{Sarn:schur}*{\S4}.

Similarly the Hecke integral gives yet another relation to the subconvexity problem. It is not difficult to verify that
  \begin{equation}\label{hecke} 
	\int_{1/N}^{1} f(z) \frac{dy}{y} = (2\pi)^{-1} L(\mdemi,f) + O(1).
  \end{equation}
 Thus~\eqref{folklore} would have implied the Lindel\"of hypothesis in the level aspect for $L(\Mdemi,f)$. 
Since we know that the estimate ~\eqref{folklore} doesn't hold in general and since we believe in the Lindel\"of hypothesis we see that the Hecke integral \eqref{hecke} has to carry a lot of cancellation.

We now turn to square-root cancellation heuristics based on Fourier expansion. 
	For $z=x+iy$, the tail of the Fourier expansion~\eqref{def:f} is negligible when $ny$ becomes large. Thus, setting $M:=1/y$, we have an approximation
  \begin{equation}\label{square-root} 
	f(x+iy) \approx M^{-\frac12} \sum_{n \sim M} a_n e^{2i\pi n x}.
 \end{equation}  
 The normalization~\eqref{def:f} is such that the Deligne bound reads  $\abs{a_n}\le \tau(n)$ for all $n\ge 1$. In many aspects the coefficients $a_n$ behave at random (cf. the Sato-Tate distribution of $a_p$, the sign changes, the bounds for $L$-values and character twists). A basic heuristic is to compare~\eqref{square-root} to a random trigonometric polynomial of degree $M$. With high probability the sup-norm over $x\in [0,1]$ of a random trigonometric polynomial is $\ll (M\log M)^{\frac12}$, see e.g.~\cite{BL01} and the references herein.

 For a \emph{fixed} form $f$, the right-hand side of~\eqref{square-root} is in fact $\ll_{f,\epsilon} M^{\epsilon}$ by the classical Wilton estimate. The estimate~\eqref{folklore} was equivalent to the stronger estimate $\ll_{\epsilon} (MN)^{\epsilon}$ that would be uniform in the level $N$ of $f$.

The following is a uniform version of the classical Wilton estimate~\eqref{square-root}.
\begin{proposition}
  Let $f\in S_2^*(N,\chi)$ be a primitive form with normalized coefficients $(a_m)_{m\ge 1}$ as in~\eqref{def:f}. For all integer $M \ge 1$ and $\epsilon>0$,
  \begin{equation*}
	\sum_{m\le M}^{} a_m e(mx) 
	\ll_\epsilon 
	M^{\frac12+\epsilon} \pnorm{f}_\infty.
  \end{equation*}
\end{proposition}

\begin{proof}
  We consider the Dirichlet kernel
 \[
 D_M(x+iy):=y^{\epsilon-1}\sum_{m\le M} m^\epsilon e(-mx),
 \]
whose $L^1$-norm satisfies $\int_0^1 \abs{D_M(x+iy)} \ll y^{\epsilon-1} M^\epsilon$. Then we compute
\begin{equation*}
  \begin{aligned}
	\int_{0}^\infty \int_{0}^1
	f(\alpha+ x +iy) D_M(x+iy) dxdy
	&=
	\int_{0}^\infty y\sum_{m\le M}
	a_m e(m\alpha) m^{\frac12}
	e^{-2\pi my}
	(my)^{\epsilon}\frac{dy}{y}\\
	&=
	\sum_{m\le M}
	\frac{a_m e(m\alpha)}{m^{\frac12}}
	\int_{0}^\infty y^{1+\epsilon} e^{-2\pi y} \frac{dy}{y}.
  \end{aligned}
\end{equation*}
By the triangle inequality this implies
\begin{equation*}
  \begin{aligned}
	\sum_{m\le M}
	\frac{a_m e(m\alpha)}{m^{\frac12}}
	& \ll_\epsilon M^\epsilon \int_0^\infty
	\left( \sup_{\Mim z= y} \abs{f(z)} \right)
	y^{\epsilon-1} dy 
	\\
	& \ll_\epsilon M^\epsilon \pnorm{f}_\infty 
  \end{aligned}
  \end{equation*}
The desired estimate follows by integration by parts.
\end{proof}

Conversely the Theorem~\ref{th:main} yields the following lower bound for Wilton sums.
\begin{proposition}
  There exist forms $f\in S_2^*(N,\chi)$ of arbitrary large level $N$, and $M\ge 1$, $x\in [0,1]$ such that
  \begin{equation*}
    \sum_{m\le M} a_m e(mx) \gg_\epsilon M^{\frac12} N^{\frac14-\epsilon}.
  \end{equation*}
\end{proposition}

\begin{remark*}
  Thus we have constructed arithmetic sequences $m\mapsto a_m e(mx)$ whose partial sums do not satisfy square-root cancellation by a large $N^{\frac14}$ margin. This is despite the fact that the sequence is random in many aspects, which is reminiscent of a similar phenomenon for $m\mapsto a_m e(-2\sqrt{m})$ observed by Iwaniec--Luo--Sarnak~\cite{ILS}*{App.\,C}.	
\end{remark*}

\begin{example*}
The proof below shows that one can choose $M\le y^{-1}$ if $f(x+iy)\gg N^{\frac14}$. From Theorem~\ref{th:point} we see that for a level $N=p^2$ and $\chi$ primitive the bound holds with $x=\frac{a}{p}$ and some integer $M\le p^3$. It would be interesting to have a direct approach in this particular case. 
\end{example*}

\begin{proof} Let $\alpha>0$ be such that $\abs{S(M)}\le \alpha M^{\frac12+\epsilon}$ for all integer $M\ge 1$, where
\begin{equation*}
  S(M):= \sum_{m\le M} a_m e(mx).
\end{equation*}

Using partial summation we find that 
\begin{equation*}
  \begin{aligned}
	f(x+iy) &= y \sum_{n=1}^{\infty}
	\left( S(n)-S(n-1) \right) n^{\frac12}e^{-2\pi ny}\\
	&\ll y \sum_{n=1}^{\infty} \abs{S(n)} n^{-\frac12}e^{-2\pi ny} \ll_\epsilon \alpha y^{-\epsilon}.
  \end{aligned}
\end{equation*}
Let $f\in S_2^*(N,\chi)$ be such that $\pnorm{f}_\infty \gg N^{\frac14}$ (Theorem~\ref{th:main}). Thus $\alpha \gg N^{\frac14}y^\epsilon$ and the claim follows.
\end{proof}

\section{Whittaker periods and special forms}
\label{sec:wh}

\subsection{A local invariant}\label{sec:wh:h-inv}

We introduce the following normalized local invariant attached to generic representations.
\begin{definition}\label{def:h} For a unitary generic representation $\pi$ of $\GL(2,\BmQ_p)$, let
\begin{equation*} 
  h(\pi):= \max_{g\in \GL(2,\BmQ_p)} \abs{W_\circ(g)},
\end{equation*}
 where $W_\circ$ is the newvector in the Whittaker model, normalized by $W_\circ(\Tme)=1$.
\end{definition}
	The definition is licit because we know~\cite{Cass73} that a nonzero newvector in the Whittaker model doesn't vanish at the identity.
 
\begin{example*}   
 If $\pi$ is unramified (that is $p^{c(\pi)}=1$), then $h(\pi)=1$. The invariant $h(\pi)$ is a measure of the ramification of $\pi$ in the sense that it behaves somehow similarly to the conductor $p^{c(\pi)}$. (They are different because $h(\pi)=1$ if $p^{c(\pi)}=p$).
\end{example*}

\subsection{A lower bound}
The introduction of the invariant $h(\pi)$ is motivated by the following estimate.
\begin{lemma}\label{lem:lower}
	Let $f\in S^*_2(N,\chi)$ be a primitive form and for all primes $p$, let $\pi_p$ be its local component at $p$. The following holds:
	\begin{equation*} 
	 \pnorm{f}_\infty \ge (2\pi e)^{-1} \prod_{p} h(\pi_p) .
	\end{equation*}
\end{lemma}

\begin{proof} The proof is better achieved in the adelic framework. Thus we work with the corresponding automorphic form $\varphi$ on $\GL_2(\BmQ)\SB \GL_2(\BmA)$. The image (inside $\BmR_+$) of $z\mapsto \abs{f(z)}$ is identical to the image of $g\mapsto \abs{\varphi(g)}$, thus $\pnorm{f}_\infty=\pnorm{\varphi}_\infty$. Let $\psi$ be the standard additive character of $\BmQ\SB \BmA$ which is unramified at all finite places and let $dx$ be the self-dual Haar measure on $\BmA$. Let
  \begin{equation}\label{def:W}
		W(g):=\int_{\BmQ\SB \BmA} \varphi(n(x)g)
  \overline{\psi(x)} dx,
  \quad g\in G(\BmA)
	\end{equation}
	be the corresponding Whittaker function. Since $f$ is primitive, we find that $W(\Tme)=e^{-2\pi}$.

	The Whittaker function factors into $W=W_\infty \prod_p W_p$. Here for all prime $p$, $W_p(g)$ is a newvector in the Whittaker model $\CmW(\pi_p,\psi_p)$ and $W_\infty \in \CmW(\pi_\infty,\psi_\infty)$ is proportional to the lowest weight vector. We see that we may arrange so that
\[
 W_\infty(\am{y}) =  ye^{-2\pi y},\ \text{and}\ W_p(\Tme) = 1.
\]
Thus $W_p$ is the normalized newvector denoted $W_\circ$ in the previous section.

Since $\BmQ \SB \BmA$ has volume one, we have the inequality
	\begin{equation*}
\begin{aligned}
		\pnorm{f}_\infty &\ge \sup_{g\in G(\BmA)} \abs{W(g)} \\
&= \sup_{g\in G(\BmR)} \abs{W_\infty(g)} \prod_p \sup_{g\in G(\BmQ_p)} \abs{W_p(g)}.
\end{aligned}
	\end{equation*}
 The claim now follows from Definition~\ref{def:h} and an elementary calculation at infinity which shows that the maximum is attained at $y=(2\pi)^{-1}$.
\end{proof}

The proof of Lemma~\ref{lem:lower} is similar to that of Hecke bound. In fact we can recover Hecke bound since the coefficients of $f$ satisfy the inequality
 \[ \abs{a_n} \le n^{\frac12} \prod_{p|n} h(\pi_p), \quad n\ge 1.
 \]

\subsection{Modular forms with prescribed ramification}
\label{sec:wh:pl}

\begin{lemma}\label{lem:ps}
  Let $\pi$ be an irreducible admissible representation of $\GL(n,\BmQ_p)$ whose central character $\omega_\pi$ has the same conductor. Then $\pi$ is a twist-minimal principal series representation, that is: $\pi$ of the form $\chi_1\boxplus \chi_2 \cdots \boxplus \chi_n$ with $\chi_i$ unramified for all $1\le i\le n-1$.
\end{lemma}

\begin{proof}
   Let $(V,N)$ be the representation of the Weil-Deligne group attached to $\pi$ by the local Langlands correspondence. We want to show that $N=0$ and $V$ is a direct sum of one-dimensional characters.

  Let $I$ be the inertia group of $\BmQ_p$ and $(I^u)_{u\ge 0}$ be the upper numbering filtration. Then the conductor of $\pi$ is equal to the Artin conductor of $(V,N)$, thus~\cite{book:serr68}*{Chap.\,VI}:
\[
c(\pi)=\log_p f(V,N) = \Mcodim \left( V^I \right)^{N=0} + \int_0^\infty \Mcodim V^{I^u} du. 
\]
The corresponding formula also holds for $c(\omega_\pi)=\log_p f(\det V)$. By assumption $c(\pi)=c(\omega_\pi)$, which implies $N=0$ on $V^I$ and
\[
\Mcodim V^{I^u} = \Mcodim \left( \det V \right)^{I^u} \qtext{for all $u\ge 0$.}
\]
Since $\det V$ is one-dimensional, we deduce that $\Mcodim V^{I}\in \set{0,1}$ from which the claim follows.
\end{proof}

\begin{corollary}\label{cor:newps}
  Let $\chi$ be an even primitive Dirichlet character of conductor $N$. Then all primitive forms in $S_2^*(N,\chi)$ are twist-minimal. Moreover the components at any prime $p$ are twist-minimal principal series. The $p$-th Fourier coefficients in the expansion~\eqref{def:f} are units: $\abs{a(p)}=1$.
\end{corollary}

\begin{remarks*}\begin{enumerate}[(i)] 
  \item A direct proof of the last assertion that $\abs{a(p)}=1$ may be found in~\cite{book:Miy}*{Prop.\,4.6.17}. 
  \item The condition that $\chi$ be \emph{primitive} is essential. For example if $N$ is square-free and $\chi=\Mun$ then the components at any prime $p\mid N$ are Steinberg. Thus there are local conditions on $N$ and on the conductor of $\chi$ that need to be satisfied.  
 \item Conversely if the desired local conditions are satisfied (e.g. $N\mid f(\chi)^2$ would be sufficient), then a positive proportion of forms in $S_2^*(N,\chi)$ have a component at every prime that is a principal series representation; this follows from~\cites{Weinstein09,Shin10}.
\end{enumerate}
\end{remarks*}

The Corollary~\ref{cor:newps} immediately follows from Lemma~\ref{lem:ps} with $n=2$.
We provide below two alternative proofs which are of independent interest. 
 It is instructive to see how each argument naturally points towards the same conclusion that $\pi$ is a principal series. 

\begin{proof}[Alternative proof of Lemma~\ref{lem:ps} when $n=2$.] 
We first prove that $\pi$ cannot be a twist $\eta\MSt$ of a Steinberg representation. Indeed if $\eta$ were unramified then $c(\eta\MSt)=1$, a contradiction. If $\eta$ were ramified, $c(\eta\MSt)=2c(\eta)$ which is strictly larger than $c(\omega_\pi)= c(\eta^2)$, again a contradiction.

We next prove that $\pi$ cannot be a dihedral supercuspidal representation. Indeed otherwise $c(\pi)\ge 2$ and $\pi$ would be induced from a quasi-character $\eta$ of a quadratic extension $E$. Its central character $\omega_\pi$ would be equal to $\eta|_{\BmQ_p^\times}\chi_E$ where $\chi_E$ is the quadratic character attached to $E$. 
If $E$ were unramified then $c(\pi)=2c(\eta)$ which is strictly larger than $c(\omega_\pi)=c(\eta|_{\BmQ_p^\times})$, a contradiction.
If $E$ were tamely ramified then $c(\pi)=c(\eta)+1$ which is strictly larger than $c(\eta|_{\BmQ_p^\times})\ge c(\omega_\pi)$, again a contradiction.
If $E$ were widely ramified then $p=2$ and $c(\pi)=c(\eta)+2$ which is strictly larger than $c(\eta|_{\BmQ_p^\times}\chi_E)$, again a contradiction. The case of non-dihedral supercuspidal representations (when $p=2$) follows from~\cite{Tunnell:llc}.

Thus $\pi$ is a principal series representation $\chi_1\boxplus \chi_2$. We have $c(\pi)=c(\chi_1)+c(\chi_2)$. On the other hand $\omega_\pi=\chi_1\chi_2$ which implies $c(\omega_\pi) \le \max(c(\chi_1),c(\chi_2))$ and thus $c(\chi_1)$ or $c(\chi_2)$ is equal to zero. The claim follows.
\end{proof}

\begin{proof}[Second alternative proof of Lemma~\ref{lem:ps} when $n=2$.]
Let $r=c(\pi)=c(\omega_\pi)$. Recall~\cite{Cass73} that the newvector in the representation of $\pi$ is stabilized by the congruence subgroup $I_r$ consisting of matrices $\mdede{a}{b}{c}{d}$ with $p^r|c$ and transforms via the character $\mdede{a}{b}{c}{d}\mapsto \omega_\pi(d)$.\footnote{There is a typo in~\cite{Cass73}*{Eq.\,(1.3)} where \Lquote{$a'$} should read \Lquote{$d'$}.}

Let $\tau$ be the representation of $\GL(2,\BmZ_p)$ induced from this character of $I_r$. 
By assumption $c(\omega_\pi)=r$, which implies that $\tau$ is irreducible.
By Frobenius reciprocity the restriction $\pi|_{\GL(2,\BmZ_p)}$ contains $\tau$.

  By a result of Henniart~\cite{Henn02:type} the representation $\tau$ is a type for the Bernstein component of twist-minimal principal series with central character equal to $\omega_\pi$ on $\BmZ_p^\times$.
  (Since $\omega_\pi$ is of conductor $p^r$ we are away from the exceptional cases in~\cite{Henn02:type}*{\S\,A.1.5} where a Bernstein component does not possess a type).
  This implies that $\pi|_{\GL(2,\BmZ_p)}$ contains $\tau$ if and only if $\pi=\chi_1\boxplus \chi_2$ with $\chi_1$ unramified and $\chi_2|_{\BmZ_p^\times}=\omega_\pi|_{\BmZ_p^\times}$, as claimed.  
\end{proof}

\subsection{The case of oldforms.}\label{conj:newforms}
The space of oldforms of level $N$ is spanned by functions $g(z)=f(dz)$ where $f$ is a primitive form of level $R$ strictly dividing $N$ and $d\mid\frac{N}{R}$. Since $\pnorm{g}_\infty=\pnorm{f}_\infty$, the bound~\eqref{folklore} for the primitive form $f$ would imply the same bound for the oldform $g$. For example the oldforms of level $N$ that come from a level $1$ form are trivially uniformly bounded. Also character twists would be compatible with~\eqref{folklore} since for all Dirichlet characters $\eta$, we have $\pnorm{f\otimes \eta}_\infty = \pnorm{f}_\infty$.

\section{Ramified Whittaker functions}\label{sec:gauss}

\subsection{A formula for the Whittaker newvector}
Let $\chi$ be a unitary character of $\BmQ^\times_p$ of conductor $p^2$.
Consider the principal series representation $\pi=\Mun\boxplus \chi$ of $G=\GL(2,\BmQ_p)$. Let $\psi$ be an unramified additive character and $W_\circ$ be the newvector in the Whittaker model of $\pi$, normalized by $W_\circ(\Tme)=1$. It is stable under the action of the congruence subgroup $I_2$ of matrices $\mdede{a}{b}{c}{d}\in \GL(2,\BmZ_p)$ such that $p^2\mid c$. For $i\in \set{0,1,2}$, let $\Tmk_i:=\km{p^i}$. 
\begin{proposition}\label{prop:diag}
	For all $y\in \BmQ_p^\times$, $W_\circ(\am{y})=\abs{y}^{\frac12}\Mun_{\BmZ_p}(y)$ and
	\[
	W_\circ(\am{y}\Tmk_0)=p^{-1}\abs{y}^{\frac12}\chi(-y)\psi(y)\epsilon(\mdemi,\chi,\psi)\Mun_{p^{-2}\BmZ_p}(y).\]
\end{proposition}

\begin{proof} The first identity is well-known~\cites{Cass73,book:moduII:deligne}. The second identity follows from the Jacquet--Langlands functional equation and may also be established in the same way as Proposition~\ref{prop:offdiag} below. We omit the details which are not directly relevant to the proof of Theorem~\ref{th:main}.
\end{proof}

\begin{proposition}\label{prop:offdiag}
	If $y\in p^{-2}\BmZ_p^\times$ is such that
\begin{equation*} 
	\chi(1-z)=\psi(yz), \quad \forall z\in p\BmZ_p,
\end{equation*}
then $W_\circ(\am{y}\Tmk_1)=p^{\frac12}$. Otherwise  $W_\circ(\am{y}\Tmk_1)=0$.
\end{proposition}

\begin{remark*} 
   Let $B$ be the Borel subgroup of upper-triangular matrices. Then $\set{\Tmk_0,\Tmk_1,\Tmk_2}$ are representatives for the double quotient $B\SB G/I_2$. Thus we have determined all the values of $W_\circ$ because any element $g\in G$ can be written as
\[
g= z \Mdede{y}{x}{0}{1}\Tmk_i \Mdede{a}{b}{c}{d}
\]
with $x\in \BmQ_p$, $y,z\in \BmQ_p^\times$, $i\in\set{0,1,2}$ and $\mdede{a}{b}{c}{d}\in I_2$, in which case
\[
W_\circ(g) = \psi(x) \chi(dz) W_\circ(\am{y}\Tmk_i).
\]
  \end{remark*}

\begin{proof} We first recall the result of Casselman~\cite{Cass73} who shows that the newform $f_\circ$ in the induced model of $\Mun\boxplus \chi$ is supported on $B\cdot I_2$.
   This determines $f_\circ:G\to \BmC$ entirely up to a multiplicative constant, which we normalize by the condition $f_\circ(\Tme)=1$,
  
  Then we shall use the fact\footnote{This follows by writing $\Tmw \nm{x} \Tmk_1=\mdede{p}{1}{p(1+x)}{x}=\mdede{-x^{-1}p}{1}{0}{x}\mdede{1}{0}{p(1+x^{-1})}{1}$.} that
\begin{equation*} 
  f_\circ (\Tmw \nm{x} \Tmk_1)=p^{-\frac12} \abs{x}^{-1}\chi(x), \qtext{if $x\in -1+p\BmZ_p$,}
\end{equation*}
and is zero otherwise, and also
\begin{equation*} 
  f_\circ(\Tmw \nm{x})= \abs{x}^{-1}\chi(x), \qtext{if $v(x)\le -2$,}
\end{equation*}
and is zero otherwise. 

The Jacquet integral gives an intertwinning from the induced model to the Whittaker model $\CmW(\pi,\psi)$. Thus letting
\begin{equation*} 
  \label{eq:jacquet-int}
  W(g):= \int_F f_\circ(\Tmw\mdede{1}{x}{0}{1}g)\overline{\psi(x)}dx, \quad g\in G,
\end{equation*}
it follows that $W_\circ(g)=W(g)/W(\Tme)$.

 We find that
  \begin{equation*} 
  W(\Tme)=\int_{v(x)\le -2}\chi(x)\overline{\psi(x)}
  \frac{dx}{\abs{x}}=\epsilon(1,\chi^{-1},\overline \psi),
\end{equation*}
and on the other hand,
\begin{equation*}
  \begin{aligned}  
	W(\mdede{y}{0}{0}{1}\Tmk_1)&=\chi(y) \abs{y}^{\frac12}
    \int_F f_\circ(\Tmw\mdede{1}{x}{0}{1}\Tmk_1)\overline{\psi(xy)}dx\\
	&=p^{-\frac12}\abs{y}^{\frac12}\chi(y)
	\int_{-1+p\BmZ_p} \chi(x) \overline{\psi(xy)}dx.
  \end{aligned}
  \end{equation*}
The proposition follows using known identities on Gauss sums and epsilon factors.
\end{proof}

The following two early observations served to indicate that the values of the Whittaker function outside of the diagonal could play a role in the context of disproving the conjecture~\eqref{folklore}. First observation is that we need to go beyond the relation between the sup-norm problem and the subconvexity problem for $L$-functions: 
	in the theory of $L$-functions one integrates the Whittaker function on the diagonal $\mdede{*}{0}{0}{*}$ and on the antidiagonal $\mdede{0}{*}{*}{0}$ (cf. the Jacquet--Langlands functional equation and the Hecke integral~\eqref{hecke}); the diagonal is included in $B I_2$ while the antidiagonal is included in $B\Tmk_0I_2$. Second observation (already seen in~\S\ref{sec:intro:ideas}) is that we need to investigate cusps which are not conjugate to $i\infty$ by the group of Atkin--Lehner involutions; the cusps conjugate to $i\infty$ correspond to the double cosets $B I_2$ and $B\Tmk_0I_2$ again. Thus we were led to study the Whittaker function on $B\Tmk_1 I_2$ which is the complement of $BI_2 \cup B\Tmk_0 I_2$ in $G$. This was the underlying motivation of the Proposition~\ref{prop:offdiag}.

\begin{corollary}\label{cor:hprincipal}
  If $\chi_1$ is unramified and $\chi_2$ has conductor $p^2$, then $h(\chi_1\boxplus \chi_2)=p^{\frac12}$.
\end{corollary}
\begin{proof}
  Let $\chi:=\chi_1^{-1}\chi_2$ which has conductor $p^2$.
  We have $\chi_1\boxplus \chi_2 \simeq \chi_1(\Mun \boxplus \chi)$. Let $W_\circ$ be the normalized newvector in the Whittaker model of $\Mun \boxplus \chi$ as above. 
  The function $W^{\chi_1}_\circ(g):= \chi_1(\det g) W_\circ(g)$ is the newvector in the Whittaker model of $\chi_1\boxplus \chi_2$. Indeed this follows from the fact that
 \[	W^{\chi_1}_\circ(gk)=\chi_1\chi_2(d) W^{\chi_1}_\circ(g), \quad
	\]
	for all elements $g\in G$ and $k=\mdede{*}{*}{*}{d}$ in $I_2$. 
	Thus $h(\pi_1\boxplus \chi_2)=h(\Mun\boxplus \chi)$.
	
The claim follows since Proposition~\ref{prop:diag} and Proposition~\ref{prop:offdiag} together imply that $h(\Mun \boxplus \chi)=p^{\frac12}$.   \end{proof}

In fact we can generalize the above results to characters $\chi_2$ of arbitrary conductor. The computation is similar thus we omit the proof. 
\begin{proposition}\label{prop:hpi}
If $\chi_1$ is unramified and $\chi_2$ has conductor $p^c$, then
		\begin{equation*}
			h(\chi_1\boxplus \chi_2)=p^{\Mdemi \lfloor \frac{c}{2} \rfloor}.
		\end{equation*}
\end{proposition}
This determines $h(\pi)$ for all twist-minimal principal series. More generally if both $\chi_1$ and $\chi_2$ are ramified and if $\pi$ is supercuspidal we shall present elsewhere a complete formula for $W_\circ(g)$. The formula will involve ${}_2F_1$ hypergeometric sums
which are the non-archimedean analogue of the classical Whittaker and $K$-Bessel functions.
 
\subsection{Proof of Theorem~\ref{th:main}}

Let $N=p^2$ with $p$ a large enough prime. Let $\chi$ be an even primitive Dirichlet character of conductor $p^2$ and $f\in S_2^*(N,\chi)$. From Corollary~\ref{cor:newps} the component at $p$ is a principal series representation $\chi_1 \boxplus \chi_2$ with $\chi_1$ unramified and $\chi_2$ of conductor $p^2$.

The Lemma~\ref{lem:lower} implies that $\pnorm{f}_\infty \gg h(\chi_1\boxplus \chi_2)$. Indeed the remaining local invariants are $1$ because $\pi$ is unramified outside $p$. The Corollary~\ref{cor:hprincipal} says that $h(\chi_1\boxplus \chi_2)=p^{\frac12}=N^{\frac14}$. This concludes the proof. \qed

\begin{remark*}
Interestingly, if $\chi_1$ and $\chi_2$ have conductor $p$, then $h(\chi_1\boxplus \chi_2)\asymp 1$. This shows that the condition in Theorem~\ref{th:main} that the central character $\chi$ be primitive modulo $N$ is necessary in our proof. 
\end{remark*}

\subsection{Proof of Theorem~\ref{th:levels}} Let $\chi$ be an even primitive Dirichlet character of conductor $N$. From Corollary~\ref{cor:newps}, all forms $f\in S_2^*(N,\chi)$ are such that for any prime $p\mid N$, the component $\pi_p$ at $p$ is a twist-minimal principal series. The Lemma~\ref{lem:lower} together with Proposition~\ref{prop:hpi} imply that 
\[\pnorm{f}_\infty \gg \prod_{p^c||N} h(\pi_p)= \prod_{p^c||N} p^{\frac12\lfloor \frac{c}{2} \rfloor}.
\qedhere
\]
\section{Proof of the other results}\label{sec:pf}

\subsection{Proof of Theorem~\ref{th:point} and Proposition~\ref{prop:neigh}}\label{sec:pf:point}
Let $f\in S_2^*(p^2,\chi)$ and let $\varphi$ be the automorphic form attached to $f$ with the same notation as in the proof of Lemma~\ref{lem:lower}. We have the expansion
\begin{equation}\label{phi-wh}
  \varphi(g)=\sum_{n\in \BmQ^\times} W\left( \am{n} g \right),\quad g\in \GL_2(\BmA),
\end{equation}
where $W=W_\infty \prod_p W_{\circ p}$, with $W(\Tme)=1$ and $W_{\circ p}$ the normalized Whittaker newvector of $\pi_p$.
We choose $g\in \GL_2(\BmA)$ such that $g_\infty=\mdede{y}{x}{0}{1}$, $g_p=\mdede{\frac{b}{p}}{0}{p}{1}$ and $g_v=1$ for all other places $p\neq v,\infty$. Recall that $b\in \BmZ_p^\times$ is such that $\chi(1-z)=\psi_p\bigl(\frac{bz}{p^2}\bigr)$ for all $z\in p\BmZ_p$. Since 
\[
\Mdede{n}{0}{0}{1} g_p = \Mdede{\frac{bn}{p^2}}{0}{0}{1} \Tmk_1,
\]
we see from Proposition~\ref{prop:offdiag} that the summand in~\eqref{phi-wh} is zero unless $n\in \BmZ_{\ge 1}$ and $n\equiv 1(p)$. This yields
\begin{equation*}
  \varphi(g) = y p^{\frac12} \sum_{n\equiv 1 (p)} a_n n^{\frac12} e^{2i\pi nz},\quad z=x+iy,
\end{equation*}
where $a_n$ are the normalized coefficients as in~\eqref{def:f}.

Choosing $x=0$, $y=1$, that is $g_\infty=\Tme$, we obtain
\begin{equation}\label{chooseg=e}
  \varphi(g) = p^{\frac12}e^{-2\pi} + O(e^{-2\pi p}).
\end{equation}
It remains to relate $\varphi(g)$ to the values of the classical form $f$ on $\FmH$, which we do via the strong approximation $\GL_2(\BmA)=\GL_2(\BmQ) \GL_2(\BmR)^+ K_0(p^2)$.

We have the decomposition
  \[
  g_p=\Mdede{\frac{b}{p}}{0}{p}{1} = \Mdede{\frac{1}{p}}{0}{ap}{1} k,
  \]
for some $k\in I_2$, where $a\in \BmZ^\times_p$ is such that $ab\equiv 1 (p)$. Let $\gamma:=\mdede{\frac{1}{p}}{0}{ap}{1}$ viewed as an element in $\GL_2(\BmQ)$. Then we have\begin{equation*}
  g= \gamma \cdot \gamma_\infty^{-1}g_\infty \cdot k',\quad k'\in K_0(p^2).
\end{equation*}

We compute that $\gamma_\infty^{-1} = \mdede{p}{0}{-ap^2}{1}$, thus 
\[\gamma^{-1}_\infty\cdot i=\frac{pi}{1-ap^2 i}=\frac{-1}{p^2 z_\chi}.\]
Since $z\mapsto \frac{-1}{p^2z}$ is the Atkin--Lehner involution we have that $\abs{\varphi(g)}=\abs{f(z_\chi)}$ which concludes the proof of Theorem~\ref{th:point}. \qed

For the Proposition~\ref{prop:neigh} we repeat the proof until~\eqref{chooseg=e}. We note that more generally choosing $g_\infty = \mdede{y}{x}{0}{1}$, we have $\varphi(g) \sim y p^{\frac12} e^{2i\pi z}$ as long as $y \ge p^{1-\epsilon}$ for some fixed $\epsilon>0$. Thus $\abs{\varphi(g)}\ge p^\delta$ as long as $p^{\delta-\frac12}\le y\le 1$. This spans a set of hyperbolic area at least $p^{\frac12-\delta}$ (for the measure $\dfrac{dxdy}{y^2}$ which is preserved by the action of $\gamma_\infty^{-1}$). 
\qed

\subsection{Proof of Theorem~\ref{th:hilbert}}\label{sec:pf:hilbert}
First we need a variant of Lemma~\ref{lem:lower}.
 The formula in Lemma~\ref{lem:lower} was simple because we were working purposely with primitive forms, normalized in such a way that their first coefficient $a_1=1$. We now work in the context of Hilbert--Maass forms with the $L^2$-normalization $\pnorm{f}_2=1$. For a place $v$ and $W_v\in \CmW(\pi_v,\psi_v)$ let 
\begin{equation*}
  (W_v,W_v):= \int_{F_v^\times} \abs{W_v(\am{y})}^2 \frac{dy}{y}.
\end{equation*} 

\begin{lemma}
  Let $f$ be a Hilbert--Maass newform with $\pnorm{f}_2=1$ and let $\pi$ be the automorphic representation it generates. Let $W=W_\infty\prod_{v\nmid \infty} W_{\circ v}$ the Whittaker function attached to $f$. Then
\begin{equation*}
\pnorm{f}_\infty^2 \gg_F L(1,\pi,\MAd)^{-1} 
\frac{\max |W_\infty|^2}{(W_\infty,W_\infty)} 
\prod_{v\nmid \infty}
\frac{h(\pi_v)^2 L(1,\pi_v,\MAd)}{(W_{\circ v},W_{\circ v})}
\end{equation*}
\end{lemma}
\begin{example*}
  If $\pi_v$ is unramified, then $h(\pi_v)=1$ and $\zeta_v(2)(W_{\circ v},W_{\circ v})=L(1,\pi_v,\MAd)$. Thus the product is absolutely convergent. 
\end{example*}

\begin{proof} The same argument as before based on~\eqref{def:W} yields
  \begin{equation*}
	\pnorm{f}_\infty \gg \max\limits_{g\in \GL_2(\BmA)} \abs{W(g)} = \max |W_\infty| \prod_{v\nmid \infty} h(\pi_v).
  \end{equation*}
From the theory of Rankin--Selberg integrals the following identity~\cite{Jacq:period-int} holds for some constant $c_F>0$: 
\begin{equation*}
\pnorm{f}^2_2 = c_F L(1,\pi,\MAd) (W_\infty,W_\infty) \prod_{v\nmid \infty}
\frac{(W_{\circ v},W_{\circ v})}{L(1,\pi_v,\MAd)}.
\end{equation*}
Since $\pnorm{f}_2=1$ the estimate follows.
\end{proof}

For almost all integers $k_1,\ldots, k_h, t_1,\ldots t_m \in \BmN_{\ge 1}$ and integral ideals $\FmN$, there are Hilbert-Maass newforms $f$ of level $\FmN$, archimedean type $(m,h)$ with weights $(k_1,\ldots,k_h)$ and spectral parameters in the respective intervals $[t_1-1,t_1], \ldots, [t_m-1,t_{m}]$ such that the local component $\pi_v$ at every non-archimedean place $v$ is a twist-minimal principal series representation $\chi_1\boxplus \chi_2$ with $\chi_1$ unramified.

By Iwaniec's convexity bound, we have 
\[
L(1,\pi,\MAd)\ll_\epsilon (t_1\cdots t_m)^\epsilon (k_1\cdots k_h)^\epsilon \TmN_{F/\BmQ}(\FmN)^\epsilon
\]
for all $\epsilon>0$. Assuming $\FmN$ is a square ideal, the Proposition~\ref{prop:hpi} shows that $h(\pi_v)=N_{F/\BmQ}(\FmN_v)^{\frac14}$ for all non-archimedean place $v$. We can now conclude the proof of Theorem~\ref{th:hilbert} using the following two lemmas.

\begin{lemma}
  Let $\pi_v$ be a twist-minimal principal series representation of $\GL(2,F_v)$. Then $(W_{\circ v},W_{\circ v})=1$ and $L(1,\pi_v,\MAd)=\zeta_v(1)$.
\end{lemma}
\begin{proof}
  The first equality follows from the known formula for $W_{\circ v}( \am{y})$ in Proposition~\ref{prop:diag}. The second equality follows from the fact that $L(1,\pi_v\times \tilde \pi_v)=\zeta_v(1)^2$.
\end{proof}

\begin{lemma}
Let $\pi$ be a generic unitary representation of $\GL(2,\BmR)$ and $W\in \CmW(\pi,\psi)$ be a nonzero lowest weight vector. Then
\begin{equation*}
  \frac{\max |W|}{(W,W)^{\frac12}} \asymp 
\begin{cases}
k^{\frac14},& \text{if $\pi$ is a discrete series of weight $k$,}\\
r^{\frac16}, &\text{if $\pi$ is a principal series of spectral parameter $r$.} 
\end{cases}
\end{equation*}
\end{lemma}

\begin{proof}
  Since $W$ transforms by a unitary character under $SO(2)$, we have by the Iwasawa decomposition $\max |W|=\max\limits_{y> 0} \abs{W( \am{y} )}$. We distinguish three cases.

 (i) If $\pi$ is a discrete series of weight $k$, then $W( \am{y} )=y^{\tfrac{k}{2}} e^{-2\pi y}$, hence: 
  \begin{equation*}
	\max\abs{W} = \left( \frac{k}{4\pi} \right)^{\frac{k}{2}} e^{-\frac{k}{2}},
	\quad 
	(W,W)= (4\pi)^{-k}\Gamma(k),
  \end{equation*}
and the assertion follows by Stirling formula~\cite{Xia07}.

 (ii) If $\pi$ is a principal series representation with trivial central character then
  $W( \am{y} )= y^{\frac12} K_{ir}(2\pi y)$. Hence the assertion follows by the asymptotic behavior of the $K$-Bessel function in the transition range~\cite{Sarn:Morawetz}:
  \begin{equation*}
	\max \abs{W}\asymp r^{\frac16}e^{-\frac{\pi r}{2}}, \quad 
	(W,W)= \frac{1}{4}\Gamma_\BmR(1+2ir)\Gamma_\BmR(1-2ir).
  \end{equation*}
  
 (iii) If $\pi$ is a principal series representation with nontrivial central character then
  $W( \am{y} )= W_{\frac{1}{2},ir}(4\pi y)$. A uniform asymptotic behavior may be found in~\cite{Dunster03}:
  \begin{equation*}
	W_{\kappa,ir}(y) \asymp r^{\frac{1}{6}+\kappa} 
	e^{-\frac{\pi r}{2}}
	\mathrm{Ai}\left(r^{\frac{2}{3}}\eta(\frac{y}{r},\frac{\kappa}{r})\right),
  \end{equation*}
  for all fixed $\kappa \ge 0$ and uniformly in $y\asymp r$ away from the zeros of the Airy function $\mathrm{Ai}(r^{\frac{2}{3}}\eta)$. Here $\eta(.,.)$ is a continuous function on $\BmR\times [0,\infty)$ which is increasing in the first variable. We deduce
  \begin{equation*}
	\max \abs{W}\asymp r^{\frac23}e^{-\frac{\pi r}{2}}, \quad 
	(W,W) = \frac{\pi \MIm \psi(ir)}{\sinh(2\pi r) \Gamma(ir)\Gamma(-ir)}\asymp r e^{-\pi r},
  \end{equation*}
which concludes the proof of the lemma.
\end{proof}

\subsection*{Acknowledgments} 
We thank Kevin Buzzard for comments on Corollary~\ref{cor:newps} and Zeev Rudnick for mentioning the work of Olofsson~\cites{Olof:large,Olof10}. It is a pleasure to thank Bill Casselman for his encouragements at the beginning of this project and Valentin Blomer, Farrell Brumley, Elon Lindenstrauss, Gergely Harcos, Roman Holowinsky, Atsushi Ichino, Nick Katz, Bernhard Kr\"otz, Erez Lapid, Zhengyu Mao, Andre Reznikov, Peter Sarnak and Akshay Venkatesh for helpful discussions and comments.
   This work was partially supported by a grant \#209849 from the Simons Foundation.

%
%



\def\cprime{$'$}\def\cprime{$'$}\def\cprime{$'$}\def\cprime{$'$}\def\cprime{$'$}

%
%

\end{document}